\documentclass[11pt]{amsart}
\usepackage[margin=1in]{geometry}

\usepackage{latexsym,amsfonts,amssymb,epsfig,verbatim}
\usepackage{amsmath, graphics, textcomp}
\usepackage{amsthm}
\usepackage{pinlabel}
\usepackage{enumitem}

\usepackage{tikz}

\usepackage[all,2cell]{xy}

\newcommand{\nc}{\newcommand}
\nc{\dmo}{\DeclareMathOperator}
\nc{\nt}{\newtheorem}

\theoremstyle{plain}
\nt{theorem}{Theorem}[section]
\nt{lemma}[theorem]{Lemma}
\nt{prop}[theorem]{Proposition}
\nt{proposition}[theorem]{Proposition}
\nt{question}[theorem]{Question}
\nt{cor}[theorem]{Corollary}
\nt{corollary}[theorem]{Corollary}
\nt{fact}[theorem]{Fact}
\nt{problem}[theorem]{Problem}
\nt{conjecture}[theorem]{Conjecture}

\newtheorem*{problem29}{Problem 2.9 (Restated)}
\newtheorem*{remark*}{Remark}

\theoremstyle{definition}
\nt{definition}[theorem]{Definition}
\nt{remark}[theorem]{Remark}
\nt{rmk}[theorem]{Remark}
\nt{example}[theorem]{Example}
\nt{exercise}[theorem]{Exercise}

\nc{\C}{\mathbb{C}}
\nc{\Q}{\mathbb{Q}}
\nc{\R}{\mathbb{R}}
\nc{\Z}{\mathbb{Z}}

\nc{\cA}{\mathcal{A}}
\nc{\cB}{\mathcal{B}}
\nc{\cC}{\mathcal{C}}
\nc{\cH}{\mathcal{H}}
\nc{\cI}{\mathcal{I}}
\nc{\cL}{\mathcal{L}}
\nc{\cR}{\mathcal{R}}
\nc{\cS}{\mathcal{S}}

\nc{\al}{\alpha}
\nc{\be}{\beta}
\nc{\de}{\delta}
\nc{\del}{\partial}
\nc{\ep}{\epsilon}
\nc{\ga}{\gamma}
\nc{\ka}{\kappa}
\nc{\la}{\lambda}
\nc{\om}{\omega}
\nc{\si}{\sigma}

\nc{\De}{\Delta}
\nc{\Ga}{\Gamma}
\nc{\GD}{\Delta}
\nc{\La}{\Lambda}
\nc{\Si}{\Sigma}

\nc{\Aut}{\mathrm{Aut}}
\nc{\Bl}{\mathrm{Bl}}
\nc{\Bu}{\mathrm{Bu}}
\nc{\bu}{\mathrm{bu}}
\nc{\Div}{\mathrm{Div}}
\nc{\ev}{\mathrm{ev}}
\nc{\GL}{\mathrm{GL}}
\nc{\Hom}{\mathrm{Hom}}
\nc{\id}{\mathrm{id}}
\nc{\im}{\mathrm{im}\,}
\nc{\law}{\mathrm{law}}
\nc{\lk}{\mathrm{lk}}
\nc{\ord}{\mathrm{ord}}
\nc{\rank}{\mathrm{rank}}
\nc{\SL}{\mathrm{SL}}
\nc{\Stab}{\mathrm{Stab}}
\nc{\Tr}{\mathrm{Tr}}
\nc{\tr}{\mathrm{tr}}
\nc{\Tree}{\mathrm{Tree}}
\nc{\wri}{\mathrm{wr}}
\nc{\UC}{\mathrm{UC}}

\nc{\ol}{\overline}
\nc{\ti}{\tilde}
\nc{\wh}{\widehat}

\nc{\x}{\times}
\nc{\ox}{\otimes}
\nc{\inn}[1]{\langle{#1}\rangle}
\nc{\pdiff}[2]{\frac{\del #1}{\del #2}}

\nc{\lto}{\longrightarrow}
\nc{\lmto}{\longmapsto}
\nc{\mto}{\mapsto}

\nc{\margin}[1]{\marginpar{\tiny #1}}

\nc{\p}[1]{\smallskip\noindent{{\bf #1}}}

\begin{document}


\title{Quasipositive Braids, Presentations, and Garside Structures}

\author{Jacob Guynee}





\begin{abstract}
We provide a presentation of the braid group in which the generating set is the set of all positive half-twists, the relations are balanced positive relations, and every relation is conjugate to one in a finite list. We then prove that for any such presentation, its positive monoid is not isomorphic to the quasipositive braid monoid. Finally, we show that the quasipositive braid monoid has no proper Garside family. These results settle a problem of Calegari from the K3 problems list.
\end{abstract}

\maketitle
\section{Introduction}

This paper addresses the following problem from the K3 problems list \cite{K3}.
\begin{quote}
\noindent\textsc{Problem 2.9.}
\textit{Is there a nice presentation of the \(n\)-stranded braid group whose
generating set is the set of all positive elementary braid half-twists?}

\medskip

\noindent\textsc{Remarks.}
\begin{enumerate}
    \item[(1)] The adjective ``nice'' here should at least mean that all
    relations consist of equating two words in the positive generators
    of the same length, and perhaps also that there be a finite number
    of local pictures describing the required relations.

    \item[(2)] The positive monoid in the standard Artin presentation
    consists of all positive braids. In the Birman--Ko--Lee presentation
    \cite{BKL}, the positive monoid is the strongly quasipositive
    braids. The positive monoid for the desired presentation would
    correspond to the quasipositive braids and would hopefully provide
    interesting algorithmic tools similar to Garside \cite{Gar69} for
    solving questions regarding quasipositivity. Similar infinite
    presentations hold for the mapping class groups of surfaces
    (Gervais \cite{Ger96}) as well as the Torelli group
    (Putman \cite{Put07}).
\end{enumerate}

\medskip

\noindent Proposed for K3 by: D.~Calegari

\noindent Scribed by: J.~Van Horn-Morris

\end{quote}

Our interpretation is that the above problem asks for a presentation satisfying three properties, which we will list below. We produce a presentation satisfying two of the properties in Theorem ~\ref{thm:pres12}, and we show that no presentation can satisfy all three in Theorem \ref{thm:nopres}. In Theorem \ref{thm:no-gars}, we provide a negative result related to remark (2): the quasipositive braid monoid admits no proper Garside family.

\textit{Quasipositive braids.}
Let $\cA_n$ be the set of isotopy classes of embedded arcs in the $n$-punctured disk $D_n$ with endpoints at distinct punctures. For $\al \in \cA_n$, let $\tau_\al\in B_n$ be the positive half-twist along $\al$. Then
\begin{equation*}
    \cH_n = \{\, \tau_\al \mid \al \in \cA_n \,\}
\end{equation*}
is the set of positive half-twists in $B_n$. The quasipositive braid monoid $QP_n$ is the monoid in $B_n$ generated by $\cH_n$. Its elements are called quasipositive braids. 

Quasipositive braids are of interest in low-dimensional topology due to the interesting properties of their closures. See the survey by Etnyre--Van Horn-Morris \cite{ETN15} for more details and for notable algorithmic problems about quasipositive braids. As mentioned in remark (2), Problem 2.9 hopes to attack some of these algorithmic problems using Garside theory. The survey by Birman--Brendle ~\cite{Birman-Brendle} explains how similar strategies have been useful for the positive braid monoid and the strongly quasipositive braid monoid. 

\textit{Formulation of the problem.}
Our first goal is to give a precise formulation of Problem 2.9. Let $\inn{\cH_n \mid R}$ be a presentation of $B_n$. The first property mentioned in remark (1) is the following.
\begin{enumerate}
    \item[\textit{(i)}]\label{item:prop1} Each relation in $R$ equates two positive words of equal length.
\end{enumerate}
Since the conjugate in $B_n$ of a positive half-twist is another positive half-twist, $B_n$ acts on $\cH_n$. Moreover, given a relation $u=v$ and a braid $g \in B_n$, we obtain a new relation
\begin{equation*}
  g\cdot u = g\cdot v.
\end{equation*}
One way to formalize the second property mentioned in remark (1) is the following.
\begin{enumerate}
  \item[\textit{(ii)}]\label{item:prop2} The relation set $R$ is the $B_n$-orbit of a finite set of relations.
\end{enumerate}

A group presentation in which all relations equate positive words also defines a monoid, called the positive monoid of the presentation. Suppose $\inn{\cH_n \mid R}$ satisfies property \textit{(i)}. The third property, given in remark (2), is the following.
  \begin{enumerate}
    \item[\textit{(iii)}]\label{item:prop3} The positive monoid of $\inn{\cH_n \mid R}$ is $QP_n$.
\end{enumerate}
While it is true that $\cH_n$ generates $QP_n$, it does not automatically follow that property \textit{(iii)} holds. Indeed, the positive monoid of a presentation does not necessarily embed into the group of the presentation.

We call a presentation of $B_n$ with generating set $\cH_n$ a positive half-twist presentation. Our precise formulation of Problem 2.9 is the following.
  \begin{problem29}
    Is there a positive half-twist presentation of $B_n$ satisfying properties (i), (ii), and (iii)?
  \end{problem29}

\textit{Answer to the problem.}
As mentioned above, a positive answer to Problem 2.9 would hopefully provide a presentation that one could use to attack algorithmic problems involving $QP_n$. Our first theorem is a negative answer to the problem.
\begin{theorem}\label{thm:nopres}
  For $n \geq 3$, there is no positive half-twist presentation of $B_n$ satisfying properties (i), (ii), and (iii).
\end{theorem}

\textit{The Extended Birman--Ko--Lee presentation.}
Our next goal is to show that there is a positive half-twist presentation of $B_n$ that satisfies properties \textit{(i)} and \textit{(ii)} (but not \textit{(iii)}). The presentation we produce extends the well-known Birman--Ko--Lee (BKL) presentation \cite{BKL} of the braid group. 

We now describe the generators and relations of the classical BKL presentation. Arrange the punctures of $D_n$ horizontally. For each $1\leq i<j\leq n$, let $\si_{ij}$ be the positive half-twist along the arc that passes from puncture $i$ to puncture $j$ by going above all the punctures in between. Let $R_{far}$ be the $B_n$-orbit of the relation
\begin{equation*}
  \si_{12} \si_{34}=\si_{34}\si_{12}.
\end{equation*}
Let $R_{triangle}$ be the union of the $B_n$-orbits of the two relations
\begin{equation*}
  \si_{12} \si_{13}=\si_{13}\si_{23}=\si_{23}\si_{12}.
\end{equation*}
We call these sets $R_{far}$ and $R_{triangle}$ because the relations they contain are known as far commutativity and triangle relations, respectively. The classical BKL presentation has generators $\si_{ij}$, $1 \leq i<j \leq n$ and all the relations of $R_{far}$ and $R_{triangle}$ that use these generators. 

With the classical presentation in hand, we define the extended BKL presentation
\[\langle \cH_n \mid R_{far} \cup R_{triangle} \rangle.\]
While this presentation satisfies properties \textit{(i)} and \textit{(ii)}, it is not obvious that this is a presentation of $B_n$. Our second theorem says that it is.

\begin{theorem}\label{thm:pres12}
    The extended BKL presentation is a positive half-twist presentation of $B_n$ satisfying properties (i) and (ii).
\end{theorem}

\begin{remark*}
\textnormal{One can also attempt to define a positive half-twist presentation extending the Artin presentation of the braid group. However, this quickly fails. The presentation would define an infinite rank Artin group. Such a group is infinitely generated and therefore cannot be the braid group.}
\end{remark*}

It follows from Theorem \ref{thm:nopres} that the positive monoid $M$ of the extended BKL presentation is not $QP_n$. In other words, the natural map $M \to QP_n$ is not injective. In order to use this presentation to study $QP_n$, one would have to understand the kernel of this map. The proof of Theorem \ref{thm:nopres} shows how to produce pairs of elements of $M$ of arbitrary length that map to the same element of $QP_n$.

\textit{Garside structures and Garside families.}
As discussed above, remark (2) suggests we try to study $QP_n$ with Garside-like techniques. The most natural question to ask is whether $QP_n$ admits a Garside structure. Garside structures were introduced by Garside \cite{Gar69} and generalized to Garside families by Dehornoy--Digne--Michel \cite{DehornoyDigneMichel}. See the survey by Birman--Brendle \cite{Birman-Brendle} for exposition on the role of Garside structures in studying the positive and strongly quasipositive monoids in the braid group. A Garside structure on a monoid $M$ consists of a special element $\Delta\in M$ such that the set of divisors $\Div(\Delta)$ satisfies certain properties. A Garside family is a subset $S \subset M$ that satisfies the same properties as $\Div(\Delta)$ but does not necessarily come with a special element of $M$. The following theorem says that there is no proper Garside family in $QP_n$ when $n \geq 3$. In particular, it admits no Garside structure.

\begin{theorem}\label{thm:no-gars}
    For $n \geq 3$, the only Garside family in $QP_n$ is $QP_n$.
\end{theorem}

\textit{Acknowledgements.}
I wish to thank my advisor, Dan Margalit, for many helpful discussions about this project. I would also like to thank John Etnyre and Danny Calegari for helpful discussions. The author was partially supported by the Department of Education through the Graduate Assistance in Areas of National Need (GAANN) fellowhsip.

\textit{AI statement.}
All of the proofs in this paper were completed in conversation with ChatGPT. As discussed above, the precise formulation of Problem 2.9 was left open to interpretation. The author considered various approaches and directed ChatGPT to find proofs of potential lemmas. Ultimately we settled on Theorems \ref{thm:nopres} and \ref{thm:pres12}. 

Theorem \ref{thm:no-gars} does not address Problem 2.9 directly, but certainly addresses remark (2). At the direction of the author, ChatGPT found various generalizations of Garside structures, and these were discussed until the author decided that ``Garside family" was the best notion to rule out for $QP_n$. The proof of Theorem \ref{thm:no-gars} was found by ChatGPT.

The exposition is due to the author. This includes identifying overcomplications in the initial proofs and bringing the intuition behind the proofs to the forefront.
\section{No ``Nice" Presentation}\label{sec:no-nice}

In this section we prove Theorem \ref{thm:nopres}. Let 
\begin{equation*}
    G= \inn{\cH_n \mid R}
\end{equation*}
be a presentation of $B_n$. Let $M$ be the positive monoid of this presentation. We will prove that if the presentation satisfies properties \textit{(i)} and \textit{(ii)}, then its positive monoid $M$ is not $QP_n$. More specifically, we show that the natural map $M \to QP_n$ is not injective.

Suppose our presentation satisfies property \textit{(i)}. Then all words representing a given element $x$ of $M$ have the same length, which we call the \textit{length} of $x$. We call a relation equating words of length $k$ a \textit{length $k$ relation}. The key insight is that length $k$ words can only be equated to other length $k$ words via relations of length at most $k$. We will prove the following lemma about $M$ and then show that this obstructs property \textit{(iii)}. 

\begin{lemma}\label{lem:prop12thenlengthk}
    Suppose $\inn{\cH_n \mid R}$ satisfies properties \textit{(i)} and \textit{(ii)}. Then for every $k$, there are length $k$ elements of $M$ that have a unique word representative.
\end{lemma}

\begin{proof}
    Suppose $R=B_n \cdot R_0$ for some finite set $R_0$ of relations. Let $\al_1,\dotsc,\al_m$ be the arcs whose positive half-twists appear in $R_0$. There is a maximum $N$ of the geometric intersection number of any pair $\al_i$ and $\al_j$.
    
    Choose arcs $\beta_1,\dotsc,\beta_k$ with pairwise geometric intersection numbers all larger than $N$. Then no relation can be used to relate $\tau_{\beta_1}\dots \tau_{\beta_k}$ to an equivalent word, so it is unique in its word equivalence class.
\end{proof}

\begin{proof}[Proof of Theorem \ref{thm:nopres}]
  By Lemma \ref{lem:prop12thenlengthk}, there is some element $g\in M$ with a unique word representative $w$. We will produce an element $g'\in M$ such that $g$ and $g'$ map to the same quasipositive braid under the natural map $M\to QP_n$. 
  
  Let $\al,\beta\in\cA_n$ be any pair of arcs. In $B_n$, we have the relation
  \begin{equation*}
    \tau_{\al} \tau_{\beta} \tau_{\al}^{-1}= \tau_{\ga},
  \end{equation*}
  where $\ga=\tau_\al(\beta)$. We can rewrite this as a relation in $QP_n$.
  \begin{equation*}
    \tau_{\al} \tau_{\beta}=\tau_{\ga}\tau_{\al}.
  \end{equation*}
  Let $w'$ be obtained from $w$ by applying this relation to a pair of adjacent generators in $w$. Let $g'=[w']\in M$. Then $g$ and $g'$ are distinct elements of $M$ that map to the same quasipositive braid in $QP_n$.  
\end{proof}
\section{A Positive Half-Twist Presentation}\label{sec:Presentation}

In this section, we prove Theorem \ref{thm:pres12}. Let
\begin{equation*}
  G= \inn{\cH_n \mid R_{far} \cup R_{triangle}}.
\end{equation*}
To avoid confusion, we use the generators $x_{\al}$ for elements of $G$ and $\tau_\al$ for elements of $B_n$. We define a map $\psi : B_n \to G$ as follows. For each $1 \leq i<j\leq n$, the BKL generator $\si_{ij}$ is the half-twist along some arc $\al_{ij}$. In other words, 
    \begin{equation*}
      \si_{ij}=\tau_{\al_{ij}}.
    \end{equation*}
    We define $\psi$ using the BKL presentation of $B_n$. For each generator $\si_{ij}$, we set
\begin{equation*}
  \psi(\si_{ij})=x_{\al_{ij}}.
\end{equation*}
This is a homomorphism since the triangle and far commutativity relations hold in $G$. 

In the reverse direction we define $\phi : G \to B_n$ by sending $x_{\al}$ to $\tau_\al$ for every $\al \in \cA_n$. This is a homomorphism since the triangle and far commutativity relations hold in $B_n$. 

At this point we know that 
\begin{equation*}
  \phi \circ \psi (\si_{ij})=\si_{ij}
\end{equation*}
for each BKL generator $\si_{ij}$. In other words, $\phi \circ \psi$ is the identity on $B_n$, and this implies $\psi$ is injective. However, we do not know that $\psi \circ \phi$ is the identity on $G$. Since the only generating set we know of for $G$ is $\cH_n$, asserting that $\psi\circ\phi$ is the identity would require showing that $\psi(\tau_{\al})=x_{\al}$ for every $\al \in \cA_n$. However, we only know this to be true for $\al=\al_{ij}$. To complete the proof, we will prove the equivalent statement that $\psi$ is surjective. The two key ingredients of the proof are an observation about the triangle relation and a notion of complexity for arcs.

\textit{The triangle relation.}
The first observation is that if two of the generators involved in a triangle relation lie in $\psi(B_n)$, then so does the third. Indeed, from the triangle relations
\begin{equation*}
  x_{\al}x_{\beta}=x_{\beta}x_{\ga}=x_{\ga}x_{\al}
\end{equation*}
one obtains
\begin{equation*}
  x_{\al}=x_{\beta}x_{\ga}x_{\beta}^{-1}, \qquad x_{\beta}=x_{\ga}x_{\al}x_{\ga}^{-1}, \qquad x_{\ga}=x_{\beta}^{-1}x_{\al}x_{\beta}.
\end{equation*}
Moreover, generators $x_{\al}$, $x_{\beta}$, and $x_{\ga}$ are part of a triangle relation if and only if the arcs $\al$, $\beta$, and $\ga$ form a triangle on $D_n$ with no interior punctures.

\textit{Complexity of an arc.}
We will define the complexity of an arc in terms of a notion of complexity for braids. First, we need to fix a presentation of $B_n$. We will use the Artin presentation, which has generators 
\begin{equation*}
  \si_i=\si_{i,i+1}
\end{equation*}
for $i=1,\dotsc,n-1$. See \cite{Birman-Brendle} for the relations. For $g \in B_n$, we denote by $|g|$ the shortest word length of $g$ with respect to the Artin presentation. For each $i=1,\dotsc,n-1$, let $\al_i$ be the straight arc between punctures $i$ and $i+1$. We are now ready to define our complexity for arcs.

Let $\al\in \cA_n$. Its complexity $C(\al)$ is the minimum of $|g|$ over all $g\in B_n$ such that $g(\al_i)=\al$ for some $i$.

\begin{proof}[Proof of Theorem \ref{thm:pres12}]
  Recall that we already observed that $\psi$ is an injective homomorphism. We will complete the proof by showing that $\psi$ is surjective. Let $\al \in \cA_n$. We will show that $x_\al\in \psi(B_n)$. We proceed by induction on $N=C(\al)$. 

  The base case is $N=0$. Since $C(\al)=0$, we must have $\al=\al_i$ for some $i$. By definition of $\psi$, we have $x_{\al}\in\psi(B_n)$. 
  
  Having completed the base case, we proceed with the inductive step. Suppose $x_{\al'}\in\psi(B_n)$ whenever $C(\al')<N$. We will show that $\al$ forms a triangle with no interior punctures with two arcs of smaller complexity. By the above observation about the triangle relations, this will complete the proof.

  Since $C(\al)=N$, there is some $g$ with $|g|=N$ and some $i$ such that
  \begin{equation*}
    \al=g(\al_i).
  \end{equation*}
  Since $|g|=N$, there is some $h$ with $|h|=N-1$ such that 
  \begin{equation*}
    g=h\si_j^{\ep},
  \end{equation*}
  where $\ep=\pm 1$. Moreover, we have $|i-j|=1$ since otherwise we would have 
  \begin{equation*}
    \si_j^{\ep}(\al_i)=\al_i,
  \end{equation*}
  and this would imply that $C(\al)<N$. 

  When $|i-j|=1$, the arcs $\al_i$, $\al_j$, and $\si_j^{\ep}(\al_i)$ form a triangle with no interior punctures. Applying $h$, we see that $\beta=h(\al_i)$, $\ga=h(\al_j)$, and $\al=g(\al_i)$ form a triangle with no interior punctures. Since $|h|=N-1$, we have that $C(\beta)$ and $C(\ga)$ are less than $N$. We have therefore proven that $\al$ forms a triangle with no interior punctures with two arcs of smaller complexity. This completes the proof.
\end{proof}

\section{No Proper Garside Families}\label{sec:NoGars}

In this section we address the possibility of using Garside theoretic techniques to study $QP_n$. We prove Theorem \ref{thm:no-gars}, which is a negative result in this direction. 

\textit{Characterizing Garside Families.}
Garside families are special subsets of a monoid. Rather than define them, we give a characterization due to Dehornoy \cite{DehornoySurvey} in terms of closure properties. We will show that if $E$ is any Garside family, then these closure properties imply $E=QP_n$. In order to give the characterization, we need to give some definitions.

For $a,b\in M$, we say that $a$ is a \textit{left divisor} of $b$ and write $a\leq b$ if there is some $c\in M$ such that
\begin{equation*}
    b=ac.
\end{equation*}
We write $a < b$ when $a \leq b$ but $a \neq b$. We define \textit{right divisors} similarly, but the notation $\leq$ and $<$ will always be for left divisors. We say $E \subset M$ is closed under right divisors if whenever $b\in E$, every right divisor $a$ of $b$ also belongs to $E$.

We say that a subset $E\subset M$ is \textit{closed under right-comultiples} if for any $s,t\in E$, if $s\le g$ and $t\le g$ for some $g\in M$, then there is some $r\in E$ such that
\begin{equation*}
    s,t\le r\le g.    
\end{equation*}

The next notion is that of a left-cancellative, right-Noetherian monoid. Every submonoid of a group, in particular $QP_n$, is left-cancellative. A left-cancellative monoid is called right-Noetherian if there are no infinite, bounded ascending sequences with respect to strict left-divisibility. That is, there is no infinite sequence of the form
\begin{equation*}
  g_1 < g_2 < g_3<\dots < g.
\end{equation*}
To prove $QP_n$ is right-Noetherian, we introduce the length homomorphism. There is a monoid homomorphism $\ell : QP_n \to \Z_{\geq0}$ defined by $\ell(\tau_\al)=1$ for any $\al \in \cA_n$. For $g \in QP_n$, we call $\ell(g)$ the \textit{length} of $g$. It follows that if $a,b \in QP_n$ satisfy $a < b$, then $\ell(a)< \ell(b)$. Since $\ell(g)$ is finite for all $g\in QP_n$, $QP_n$ is right-Noetherian.

Finally, we can give the characterization of Garside families. Let $M$ be a left-cancellative, right-Noetherian monoid with no invertible elements except $1$. Let $E \subset M$. Dehornoy proved \cite[Corollary 3.19]{DehornoySurvey} that $E$ is a Garside family if and only if it satisfies the following properties. 
\begin{enumerate}[label=(\alph*)]
  \item $1 \in E$, 
  \item $E$ generates $M$, 
  \item $E$ is closed under right-divisors, 
  \item $E$ is closed under right-comultiples.
\end{enumerate}

These are satisfied when $E=M$. In particular, $QP_n$ is a Garside family in $QP_n$. It remains to show that it is the only Garside family in $QP_n$. The following lemma will allow us to use the right-comultiple condition to inductively place all elements of $QP_n$ into $E$ according to their length.
\begin{lemma}\label{lem:right comultiple ind condition}
  Suppose $E \subset QP_n$ is closed under right-comultiples. Let $s$ and $t$ be distinct elements of $E$ with $\ell(s)=\ell(t)=d$. Suppose $g$ is a common right multiple of $s$ and $t$ and $\ell(g)=d+1$. Then $g \in E$.
\end{lemma}

\begin{proof}
  Since $g$ is a common right multiple of $s$ and $t$ and $E$ is closed under right-comultiples, there is some $r \in E$ such that
    \begin{equation*}
      s, t \leq r \leq g.
    \end{equation*}
    Since $s$ and $t$ are distinct, we cannot have $\ell(r)=d$. Otherwise, we would have $s=t=r$. Therefore we must have $\ell(r)=d+1$ and $r=g$. Since $r\in E$, this proves $g \in E$. 
\end{proof}

\begin{proof}[Proof of Theorem \ref{thm:no-gars}.]
  Let $E \subset QP_n$ be a Garside family. Let $g \in QP_n$. We will show by induction on its length $\ell(g)$ that $g\in E$, and this will complete the proof. The base cases are lengths 0 and ~1. Suppose $g\in QP_n$ with $\ell(g)=0$. Then $g$ must be the identity, which lies in $E$ due to item (a) of the characterization. 

The next base case is $\ell(g)=1$. Then $g=\tau_{\al}$ for some $\al\in\cA_n$. Since $E$ generates $QP_n$ by item (b), there are elements $s_1,\dotsc, s_r\in E$ such that
\begin{equation*}
  \tau_\al = s_1\dots s_r.
\end{equation*}
Applying the length homomorphism we obtain
\begin{equation*}
  1=\ell(s_1)+\dots + \ell(s_r).
\end{equation*}
Since $\ell$ is nonnegative, we have $\ell(s_i)=1$ for some $i$ and $\ell(s_j)=0$ for $j\neq i$. Therefore we have
\begin{equation*}
  \tau_{\al}=s_i.
\end{equation*}
Since $s_i\in E$, this proves $\tau_{\al}\in E$.

Next we carry out the inductive step. Let $g\in QP_n$ have length $k \geq 2$. Assume all length $k-1$ elements of $QP_n$ lie in $E$. We will prove $g \in E$.

There are two cases, one of which must be proven before the other. The first case is that there are distinct $x,y \in \cH_n$ such that
\begin{equation*}
  g = u x y v
\end{equation*}
with $u,v \in QP_n$. The second case is that $g=x^k$ for some $x\in\cH_n$. We begin with the former case. Let 
\begin{equation*}
  s= uxv, \qquad t = uyv.
\end{equation*}
Moreover, $s$ and $t$ are distinct by cancellativity. By induction, we have $s,t \in E$. We have
\begin{equation*}
  g=s(v^{-1}yv)=t(v^{-1}y^{-1}xyv).
\end{equation*}
Since $v^{-1}yv$ and $v^{-1}y^{-1}xyv$ are positive half-twists, $g$ has length one larger than $s$ and $t$. Since $g$ is also a common right multiple of $s$ and $t$, we can apply Lemma \ref{lem:right comultiple ind condition} and we obtain $g \in E$, as desired. 

Now we consider the case that $g = x^k$ for some $x \in \cH_n$. There is some arc $\al$ such that $x=\tau_\al$. Choose arcs $\beta$ and $\ga$ such that $\al$, $\beta$, and $\ga$ form a triangle. Let $y=\tau_\beta$ and $z=\tau_\ga$. Then the triangle relation gives
\begin{equation*}
  yx=zy.
\end{equation*}
Now we set
\begin{equation*}
  w=x^{k-1}yx = x^{k-1}zy.
\end{equation*}
Then $w$ has length $k+1$ and is a common right multiple of $x^{k-1}y$ and $x^{k-1}z$, which are distinct and of length $k$. Since $x$, $y$, and $z$ are distinct, the length $k$ elements $x^{k-1}y$ and $x^{k-1}z$ belong to the previous case and therefore lie in $E$. By Lemma \ref{lem:right comultiple ind condition}, we get $w \in E$. Finally, we notice that
\begin{equation*}
  w= \left(x^{k-1}y x^{-(k-1)}\right) x^k.
\end{equation*}
Since $E$ is closed under right divisors and $w \in E$, we have $x^k \in E$, as desired.

By induction, we have proven that $E=QP_n$. This completes the proof.
\end{proof}

\bibliographystyle{plain}
\bibliography{template}

\end{document}